\documentclass[reqno,11pt]{amsart}
\usepackage{epsfig,amscd,amssymb,amsmath,amsfonts}
\usepackage{bigdelim}
\usepackage{amsmath}
\usepackage{graphicx}
\usepackage{amsthm,color}
\usepackage{verbatim}
\usepackage{tikz}
\usepackage{lscape}
\usetikzlibrary{graphs}
\usetikzlibrary{graphs,quotes}
\usetikzlibrary{decorations.pathmorphing}

\tikzset{snake it/.style={decorate, decoration=snake}}
\tikzset{snake it/.style={decorate, decoration=snake}}

\usetikzlibrary{decorations.pathreplacing,decorations.markings,snakes}

\newtheorem{theorem}{Theorem}[section]
\newtheorem{lemma}[theorem]{Lemma}
\newtheorem{proposition}{Proposition}[section]
\theoremstyle{definition}
\newtheorem{definition}[theorem]{Definition}
\newtheorem{corollary}[theorem]{Corollary}
\newtheorem{conjecture}[theorem]{Conjecture}

\newtheorem{example}[theorem]{Example}

\theoremstyle{remark}
\newtheorem{remark}[theorem]{Remark}

\numberwithin{equation}{section}

\usepackage[margin=1.05in]{geometry}
\usepackage[colorlinks]{hyperref}
\usepackage{siunitx}

\usepackage{blkarray, bigstrut}
\usepackage{gauss}

\makeatletter
\let\@wraptoccontribs\wraptoccontribs
\makeatother

\begin{document}

\title{The Algebra of $S^2$-Upper Triangular Matrices}
%    Information for first author

\author{Steven R. Lippold}
\address{Department of Mathematics, Taylor University, Upland, IN 46989}
\email{steve\_lippold@taylor.edu}

%    General info


\begin{abstract} Based on work presented in \cite{edge}, we define $S^2$-Upper Triangular Matrices and $S^2$-Lower Triangular Matrices, two special types of $d\times d(2d-1)$ matrices generalizing Upper and Lower Triangular Matrices, respectively. Then, we show that the property that the determinant of an Upper Triangular Matrix is the product of its diagonal entries is generalized under our construction. Further, we construct the algebra of $S^2$-Upper Triangular Matrices and give conditions for an LU-Decomposition with $S^2$-Lower Triangular and $S^2$-Upper Triangular Matrices, respectively.
\end{abstract}

\maketitle

%\section*{This is an unnumbered first-level section head}

%
%%%%%%%%%%%%%%%%%%%%%%%%%%%%%%%%%%%%%%%%%%%%%%%%%%%

%%%%%%%%%%%%%%%%%%%%%%%%%%%%%%%%%%%%%%%%%%%%%%%%%%%%
\section{Introduction}

One common type of special matrix with a variety of applications is that of the Upper Triangular matrix. Notably, it is well known that if $A$ is an Upper Triangular matrix, then the determinant $det(A)$ is the product of its diagonal entries. In addition, matrices that satisfy certain conditions can be factored in terms of the product of an upper and lower triangular matrix (so called LU Factorization), which is often more desirable computationally. 

In \cite{sta2}, a generalization of the determinant map, denoted $det^{S^2}$, was first constructed and has been subsequently studied further in \cite{SDiss}, \cite{edge}, and \cite{dets2}. Namely, if $V_d$ is a vector space of dimension $d$ over a field $k$, there exists a map $det^{S^2}:V_d^{\otimes d(2d-1)}\to k$ with the property that $det^{S^2}(\otimes_{1\leq i<j\leq 2d}(v_{i,j}))=0$ if there exists $1\leq x<y<z\leq 2d$ such that $v_{x,y}=v_{x,z}=v_{y,z}$. Further, this map is unique up to a constant if $d=2$ or $d=3$ \cite{edge}. An interesting follow-up is to construct a class of matrices that generalize the idea of Upper Triangular Matrix, but in relation to the map $det^{S^2}$. In this paper, we give such a construction.

Now, we give an outline to this paper. In Section 2, we start by recalling some terminology concerning graphs before giving an equivalent definition of Triangular Matrices which is useful for this paper. Then, we give some basic results concerning Triangular Matrices and determinant map. Following this, we give some basic results from \cite{edge} and \cite{dets2} regarding a vector space $\mathcal{T}^{S^2}_{V_d}[2d]$, a special element $E^{(2)}_d\in\mathcal{T}^{S^2}_{V_d}[2d]$, and the map $det^{S^2}$.

In Section 3, we give two new constructions. First, we define the $S^2$-Diagonal of a $d\times d(2d-1)$ matrix, which generalizes the idea of the diagonal of a square matrix. Then, we define $S^2$-Triangular Matrices before giving a result connecting $det^{S^2}$, $S^2$-Upper Triangular Matrices, and the $S^2$-diagonal. 

In Section 4, we construct a $k$-algebra of $S^2$-Upper Triangular Matrices, starting first with the case $d=2$ and extending it for $d>2$. We define a multiplication on $d\times d(2d-1)$ matrices and give some conditions for an LU-Decomposition using $S^2$-Lower Triangular Matrices and $S^2$-Upper Triangular Matrices, respectively. 

\begin{comment}
\section*{Outline}
\begin{enumerate}
    \item Introduction
    \item Preliminaries
    \begin{enumerate}
        \item Basic facts of Upper Triangular
        \item Reformulation of Upper Triangular
        \item Tensor Algebra and Determinant
        \item $\mathcal{T}^{S^2}$ and $det^{S^2}$
        \item The Element $E_d$
    \end{enumerate}
    \item $S^2$-Upper Triangular Matrices
    \begin{enumerate}
        \item Definition of $S^2$-Diagonal
        \item Definition of $S^2$-Upper Triangular
        \item Properties of $S^2$-Upper Triangular
    \end{enumerate}
\end{enumerate}
\end{comment}
\section{Preliminaries}
In this paper, we assume that $k$ is a field of characteristic not 2 or 3, $\otimes=\otimes_k$, and that $V_d$ is a vector space of dimension $d$ over $k$. Given a vector $v\in V_d$, we will take $v=(v^1,v^2,\ldots,v^d)$, where $v^i\in k$ for $1\leq i\leq d$.  Lastly, we take $\mathcal{B}_d=\{e_1,e_2,\ldots,e_d\}$ to be basis for $V_d$.
\subsection{Edge $d$-Partitions of Graphs}
First, we will cover some basics of edge $d$-partitions of graphs. Throughout this paper, we will assume that graphs are undirected, do not have multiple edges or loops, and are not weighted. 

Regarding notation, we will denote the vertex set of a graph $\Gamma$ as $V(\Gamma)$ and the edge set as $E(\Gamma)$. Next, recall the following definition.

\begin{definition}[\cite{edge}]
    Let $\Gamma$ be a graph and $\Gamma_1,\Gamma_2,\ldots,\Gamma_d$ be subgraphs of $\Gamma$. Then, $(\Gamma_1,\Gamma_2,\ldots,\Gamma_d)$ is an edge $d$-partition of $\Gamma$ if 
    \begin{enumerate}
        \item For all $1\leq i\leq d$, $V(\Gamma_i)=V(\Gamma)$.
        \item For all $1\leq i<j\leq d$, $E(\Gamma_i)\cap E(\Gamma_j)$ is empty.
        \item $\displaystyle\cup_{i=1}^dE(\Gamma_i)=E(\Gamma)$.
    \end{enumerate}
    We say that an edge $d$-partition is cycle-free if each of the subgraphs $\Gamma_i$ for $1\leq i\leq d$ are cycle free.

    We say that an edge $d$-partition is homogeneous if $|E(\Gamma_i)|=|E(\Gamma_j)|$ for all $1\leq i<j\leq d$.

    We will denote the set of edge $d$-partitions of $\Gamma$ as $\mathcal{P}_d(\Gamma)$ and the set of homogeneous cycle-free edge $d$-partitions of $\Gamma$ as $\mathcal{P}_d^{h,cf}(\Gamma)$.
\end{definition}
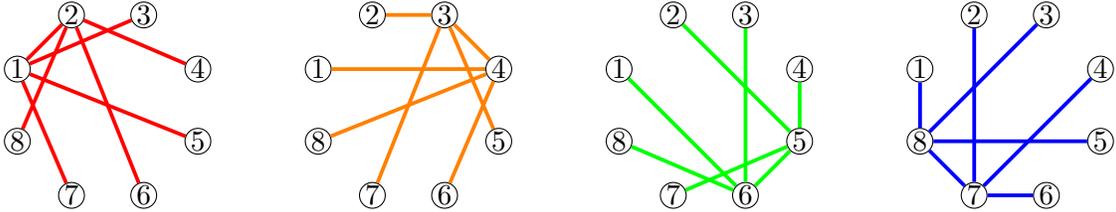
\begin{figure}[h!]
	\centering
	\begin{tikzpicture}
		[scale=0.8,auto=left,every node/.style={shape = circle, draw, fill = white,minimum size = 1pt, inner sep=0.3pt}]%baseline=(a.center)]%{circle,fill=black}]
		%\tikzset{VertexStyle/.style = {shape = circle,fill = black,minimum size = 9mm,inner sep=2pt}}
		\node (n1) at (0,2.1) {1};
		\node (n2) at (0.9,3)  {2};
		\node (n3) at (2.1,3)  {3};
		\node (n4) at (3,2.1)  {4};
		\node (n5) at (3,0.9)  {5};
		\node (n6) at (2.1,0)  {6};
		\node (n7) at (0.9,0)  {7};
		\node (n8) at (0,0.9)  {8};
		%\node (front) at (-2,0.5)   {$\Gamma_1$}
		\foreach \from/\to in {n1/n2,n1/n3,n1/n5,n1/n7,n2/n4,n2/n6,n2/n8}
		\draw[line width=0.5mm,red]  (\from) -- (\to);	
		\node (n12) at (5,2.1) {1};
		\node (n22) at (5.9,3)  {2};
		\node (n32) at (7.1,3)  {3};
		\node (n42) at (8,2.1)  {4};
		\node (n52) at (8,0.9)  {5};
		\node (n62) at (7.1,0)  {6};
		\node (n72) at (5.9,0)  {7};
		\node (n82) at (5,0.9)  {8};
		%\node (front) at (-2,0.5)   {$\Gamma_1$}
		\foreach \from/\to in {n32/n22,n32/n42,n32/n52,n32/n72,n42/n12,n42/n62,n42/n82}
		\draw[line width=0.5mm,orange]  (\from) -- (\to);	
		
		\node (n13) at (10,2.1) {1};
		\node (n23) at (10.9,3)  {2};
		\node (n33) at (12.1,3)  {3};
		\node (n43) at (13,2.1)  {4};
		\node (n53) at (13,0.9)  {5};
		\node (n63) at (12.1,0)  {6};
		\node (n73) at (10.9,0)  {7};
		\node (n83) at (10,0.9)  {8};
		%\node (front) at (-2,0.5)   {$\Gamma_1$}
		\foreach \from/\to in {n53/n63,n23/n53,n43/n53,n53/n73,n13/n63,n33/n63,n63/n83}
		\draw[line width=0.5mm,green]  (\from) -- (\to);	
		
		\node (n14) at (15,2.1) {1};
		\node (n24) at (15.9,3)  {2};
		\node (n34) at (17.1,3)  {3};
		\node (n44) at (18,2.1)  {4};
		\node (n54) at (18,0.9)  {5};
		\node (n64) at (17.1,0)  {6};
		\node (n74) at (15.9,0)  {7};
		\node (n84) at (15,0.9)  {8};
		%\node (front) at (-2,0.5)   {$\Gamma_1$}
		\foreach \from/\to in {n74/n84,n74/n24,n44/n74,n64/n74,n14/n84,n34/n84,n54/n84}
		\draw[line width=0.5mm,blue]  (\from) -- (\to);
		
		%\node[state,above of=B1] (C1) {$C_1$};
	\end{tikzpicture}
	\caption{A homogeneous cycle-free edge $4$-partition of the complete graph $K_8$.} \label{figd4}
\end{figure}
Figure \ref{figd4} gives an example of an edge $4$-partition, i.e. $d=4$. Further, each of the subgraphs are trees with $7$ edges, so the edge $4$-partition is homogeneous and cycle-free.

Another important preliminary concerning edge $d$-partitions of the complete graph $K_{2d}$. The following result is from \cite{edge}.
\begin{lemma}[\cite{edge}]\label{invol}
    Let $(\Gamma_1,\Gamma_2,\ldots,\Gamma_d)$ be a homogeneous cycle-free edge $d$-partition of $K_{2d}$ and $1\leq x<y<z\leq 2d$. Then, there exists a unique homogeneous cycle-free edge $d$-partition $(\Lambda_1,\Lambda_2,\ldots,\Lambda_d)$ of $K_{2d}$ such that $(\Gamma_1,\Gamma_2,\ldots,\Gamma_d)$ and $(\Lambda_1,\Lambda_2,\ldots,\Lambda_d)$ agree on all their edges, except on the triangle $(x,y,z)$.
\end{lemma}
\begin{remark}
    For $1\leq x<y<z\leq 2d$ and a homogeneous cycle-free edge $d$-partition of $K_{2d}$ given by $\Gamma=(\Gamma_1,\Gamma_2,\ldots,\Gamma_d)$, we will denote the unique edge $d$-partition given in Lemma \ref{invol} as $\Gamma^{(x,y,z)}$. This edge $d$-partition is known as the involution of $\Gamma$.
\end{remark}

There is a special graph which is of interest for this paper.
\begin{definition}
    Fix $1\leq b\leq d$. For the vertices $1,2,\ldots,2d$ define the sets of edges
    \[L_d(2b-1,2b)=\{(2a,2b-1)|1\leq a<b\}\cup\{(2b-1,2c-1)|b<c\leq d\},\]
    \[C_d(2b-1,2b)=\{(2b-1,2b)\},\]
    and
    \[R_d(2b-1,2b)=\{(2a-1,2b)|1\leq a<b\}\cup\{(2b,2c)|b<c\leq d\}.\]
    Let the twin star graph $TS_d(2b-1,2b)$ be the graph with vertices $1,2,\ldots,2d$ and edges
    \[E(TS_d)=L_d(2b-1,2b)\cup C_d(2b-1,2b)\cup R_d(2b-1,2b).\]
    The edges in $L_d(2b-1,2b)$ are called the left legs of the graph and the edges in $R_d(2b-1,2b)$ are called the right legs of the graph.
\end{definition}
\begin{example}
    As an example, consider the following graph $TS_7(5,6)$. Here there are 14 vertices, with 6 left legs and 6 right legs. For example, the edges $(2,5)$ and $(5,13)$ are left legs, whereas the edges $(3,6)$ and $(6,10)$ are right legs. The graph is given in Figure \ref{TwinStarGraph}.
    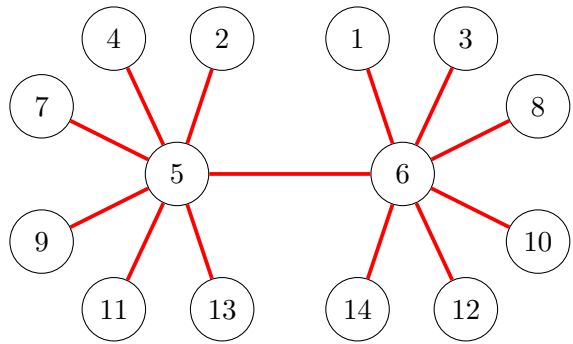
\begin{figure}[h]
\centering
\begin{tikzpicture}
  [scale=0.6,auto=left]%,every node/.style={shape = circle, draw, fill = white,minimum size = 14pt, inner sep=0.3pt}]%baseline=(a.center)]%{circle,fill=black}]
	%\tikzset{VertexStyle/.style = {shape = circle,fill = black,minimum size = 9mm,inner sep=2pt}}
	\node[shape=circle,draw=black,minimum size = 24pt,inner sep=0.3pt] (n1) at (3,0) {$5$};
	\node[shape=circle,draw=black,minimum size = 24pt,inner sep=0.3pt] (n2) at (8,0) {$6$};
	\node[shape=circle,draw=black,minimum size = 24pt,inner sep=0.3pt] (n3) at (7,3) {$1$};
  \node[shape=circle,draw=black,minimum size = 24pt,inner sep=0.3pt] (n4) at (4,3) {$2$};
	\node[shape=circle,draw=black,minimum size = 24pt,inner sep=0.3pt] (n5) at (9.4,3) {$3$};
  \node[shape=circle,draw=black,minimum size = 24pt,inner sep=0.3pt] (n6) at (1.6,3) {$4$};
	\node[shape=circle,draw=black,minimum size = 24pt,inner sep=0.3pt] (n51) at (11,1.5) {$8$};
  \node[shape=circle,draw=black,minimum size = 24pt,inner sep=0.3pt] (n61) at (0,1.5) {$7$};
	\node[shape=circle,draw=black,minimum size = 24pt,inner sep=0.3pt] (n71) at (11,-1.5) {$10$};
  \node[shape=circle,draw=black,minimum size = 24pt,inner sep=0.3pt] (n81) at (0,-1.5) {$9$};
	\node[shape=circle,draw=black,minimum size = 24pt,inner sep=0.3pt] (n7) at (9.4,-3) {$12$};
  \node[shape=circle,draw=black,minimum size = 24pt,inner sep=0.3pt] (n8) at (1.6,-3) {$11$};
	\node[shape=circle,draw=black,minimum size = 24pt,inner sep=0.3pt] (n9) at (7,-3) {$14$};
  \node[shape=circle,draw=black,minimum size = 24pt,inner sep=0.3pt] (n10) at (4,-3) {$13$};

	  \draw[line width=0.5mm,red]  (n1) -- (n2)  ;
		\draw[line width=0.5mm,red]  (n1) -- (n4)  ;
		\draw[line width=0.5mm,red]  (n1) -- (n6)  ;
		\draw[line width=0.5mm,red]  (n1) -- (n61)  ;
		\draw[line width=0.5mm,red]  (n1) -- (n8)  ;
		\draw[line width=0.5mm,red]  (n1) -- (n81)  ;
		\draw[line width=0.5mm,red]  (n1) -- (n10)  ;
		\draw[line width=0.5mm,red]  (n3) -- (n2)  ;
	  \draw[line width=0.5mm,red]  (n5) -- (n2)  ;
		\draw[line width=0.5mm,red]  (n51) -- (n2)  ;
	  \draw[line width=0.5mm,red]  (n7) -- (n2)  ;
	  \draw[line width=0.5mm,red]  (n71) -- (n2)  ;
		\draw[line width=0.5mm,red]  (n9) -- (n2)  ;
		%\draw[line width=0.5mm,dashed]  (n2) -- (n3);	
\end{tikzpicture}
\caption{The twin star graph $TS_7(5,6)$} \label{TwinStarGraph}
\end{figure}
\end{example}
\begin{definition}
    For $TS_d(2b-1,2b)$, take the set $$LL=\{x|(x,2b-1)\in L_d(2b-1,2b) \textrm{ or }(2b-1,x)\in L_d(2b-1,2b)\}$$ and $$RL=\{x|(x,2b)\in L_d(2b-1,2b) \textrm{ or }(2b,x)\in L_d(2b-1,2b)\}.$$ Notice that $LL$ and $RL$ are finite subsets of the natural numbers and there is a natural order on them.

    Given a left leg $(x,2b-1)$ or $(2b-1,x)$ of $TS_d(2b-1,2b)$, we say that the leg number is the element number of $x$ after $LL$ is enumerated by the natural order on subsets of the natural numbers. Similarly, given a right leg $(x,2b)$ or $(2b,x)$, we say that the leg number is the element number of $x$ after $RL$ is enumerated by the order. By default, let the leg number of $(2b-1,2b)$ be 1. 

    Alternatively, $$LL=\{x|x<2b-1\textrm{ and }x\textrm{ is even}\}\cup\{y|2b-1<y\textrm{ and }y\textrm{ is odd}\}.$$
    Also,
    \[RL=\{x|x<2b-1<2b\textrm{ and }x\textrm{ is odd}\}\cup\{y|2b-1<y\textrm{ and }y\textrm{ is even}\}.\]
\end{definition}
Given $(x,y)$, we will denote the leg number of $(x,y)$ as $l(x,y)$.
\begin{example}
    Take the Twin Star Graph given in Figure \ref{TwinStarGraph}. Notice that the left and right legs are in order, going from the top center to the bottom center along the outside. So, we have
    \[LL=\{2,4,7,9,11,13\},\]
    and
    \[RL=\{1,3,8,10,12,14\}.\]
    In this case, we see that $l(4,5)=2$, whereas $l(6,8)=3$. We could likewise determine the leg number of all of the other edges of $TS_7(5,6)$. Lastly, note that $l(5,6)=1$, as $C_7(5,6)=\{(5,6)\}$, which is pictured in Figure \ref{TwinStarGraph} as the edge in the center of the graph.
\end{example}
\subsection{Upper Triangular Matrices}
We now recall upper triangular matrices and reformulate in such a way that is useful for our later construction.

Given any $d\times n$ matrix $A$ with columns $c_i$, then we can associate $A$ with the simple tensor $c_1\otimes c_2\otimes\ldots\otimes c_n\in V_d^{\otimes n}$. In particular, if $e_1,e_2,\ldots,e_d$ is a basis for $V_d$, then the $d\times d$ identity matrix $I_d$ is associated with the simple tensor $e_1\otimes e_2\otimes\ldots\otimes e_d$.

\begin{definition}\label{UpperLowerDef}
    Let $A$ be a $d\times d$ matrix with columns $c_i$ and associated simple tensor $c_1\otimes\ldots\otimes c_d$. 
    
    We say that $A$ is an Upper Triangular Matrix if for all $1\leq i\leq d$ there exists $\alpha_j\in k$ for $1\leq j\leq i$ such that $c_i=\displaystyle\sum_{j=1}^{i}\alpha_je_j$.

    We say that $A$ is a Lower Triangular Matrix if for all $1\leq i\leq d$ there exists $\alpha_j\in k$ for $i\leq j\leq d$ such that $c_i=\displaystyle\sum_{j=i}^d\alpha_je_j$.

    If $A$ is Upper Triangular or Lower Triangular, we say that $A$ is a Triangular Matrix.
\end{definition}

To illustrate Definition \ref{UpperLowerDef}, we give a couple examples.
\begin{example}
First, consider the matrix $A=\displaystyle\begin{pmatrix}
    1&2&5\\
    0&6&10\\
    0&0&3
\end{pmatrix}.$ Then, $A$ is associated with the simple tensor $c_1\otimes c_2\otimes c_3$, where 
\[c_1=e_1\]
\[c_2=2e_1+6e_2\]
and
\[c_3=5e_1+10e_2+3e_3.\]
Under Definition \ref{UpperLowerDef}, this is an Upper Triangular Matrix but not a Lower Triangular Matrix.

Next, consider the matrix $B=\displaystyle\begin{pmatrix}
    1&0&0\\
    3&0&0\\
    1&2&5
\end{pmatrix}.$ Then $B$ is associated with the simple tensor $u_1\otimes u_2\otimes u_3$, where 
\[u_1=e_1+3e_2+e_3,\]
\[u_2=2e_3,\]
and
\[u_3=5e_3.\]
Under Definition \ref{UpperLowerDef}, this is not an Upper Triangular matrix, but $B$ is a Lower Triangular Matrix.

Lastly, consider the matrix $C=\displaystyle\begin{pmatrix}
    1&1&0\\
    1&0&0\\
    0&0&1
\end{pmatrix}.$ Then $C$ is associated with the simple tensor $v_1\otimes v_2\otimes v_3$, where 
\[v_1=e_1+e_2,\]
\[v_2=e_1,\]
and
\[v_3=e_3.\]
Under Definition \ref{UpperLowerDef}, this is not a Triangular Matrix (neither an Upper Triangular Matrix nor a Lower Triangular Matrix).
\end{example}

Regarding Triangular Matrices, one common use is LU-Decomposition, which is summarized in the following proposition.

\begin{proposition}[\cite{lu}]\label{LUProp}
    For any $d\times d$ matrix $A$, if all of the leading principal minors are nonzero, then there exists a Unit Lower Triangular Matrix $L$ (that is, $A$ is Lower Triangular and $a_{i,i}=1$ for $1\leq i\leq d$) and a non-singular Upper Triangular Matrix $U$ such that $A=LU$.
\end{proposition}

Lastly, we reformulate a definition for the diagonal of a matrix. 
\begin{definition}\label{DiagDef}
    Let $A$ be a $d\times d$ matrix with entries $a_{i,j}$ and $I_d$ be the identity matrix with entries $e^i_j$, where $e_j=(e^1_j,e^2_j,\ldots,e^d_j)$.
    
    Define the $S^1$-diagonal indicator set is defined $DI_d[1]=\{(i,j)|e_j^i\neq0\}$ and the $S^1$-diagonal as $diag(A)=\{a_{i,j}|(i,j)\in DI_d[1]\}.$
\end{definition}
\begin{remark}
    For a $d\times d$ matrix $A$, notice that $DI_d[1]=\{(1,1),(2,2),(3,3),\ldots,(d,d)\}$ and that $diag(A)$ agrees with the usual definition of the diagonal of a square matrix. We are formulating Definition \ref{DiagDef} in this way in order to give motivation for a later definition.
\end{remark}

\subsection{The Tensor Algebra and the Determinant Map}
Next, we will recall the Tensor Algebra and its connection to the determinant map.
\begin{definition}[\cite{bour}]
    Let $V_d$ be a vector space of dimension $d$. Then, we define the vector space $\mathcal{T}_{V_d}[n]=V_d^{\otimes n}$. The graded vector space $\mathcal{T}_{V_d}=\bigoplus_{n\geq0}\mathcal{T}_{V_d}[n]$, where $\mathcal{T}_{V_d}[0]=k$, is called the Tensor Algebra.
\end{definition}
\begin{remark}
    Notice that since every $d\times d$ matrix is associated with a simple tensor in $V_d^{\otimes d}$, it follows that every $d\times d$ matrix is associated with a simple tensor in $\mathcal{T}_{V_d}[d]$.
\end{remark}
\begin{remark}
    The Tensor Algebra is a $k$-algebra with multiplication defined by concatenation of simple tensors. This structure will not be important for our purposes in this paper. 
\end{remark}

In the context of the Tensor Algebra, we can give a connection to the determinant map.
\begin{proposition}\label{detprop}
    There exists a unique map $det:\mathcal{T}_{V_d}[d]\to k$ such that
    \begin{enumerate}
        \item $det(e_1\otimes e_2\otimes\ldots\otimes e_d)=1$.
        \item If there exists $1\leq x<y\leq d$ such that $v_x=v_y$, then $det(v_1\otimes\ldots\otimes v_d)=0$.
        \item If $c_1\otimes c_2\otimes\ldots\otimes c_d$ is the simple tensor associated with a $d\times d$ matrix $A$, then $det(c_1\otimes\ldots\otimes c_d)=det(A)$, where $det(A)$ is the usual determinant of the matrix $A$.
    \end{enumerate} 
\end{proposition}

Regarding the determinant function, recall the following result.
\begin{proposition}\label{detprop2}
    Let $A$ be an Upper Triangular Matrix. Then, $det(A)$ is precisely the product of the diagonal entries.
\end{proposition}

\subsection{The Vector Space $\mathcal{T}^{S^2}[2d]$}
Next, we recall a construction from \cite{sta2}. This construction parallels the Tensor Algebra previously discussed (see \cite{SDiss} for more information).

\begin{definition}[\cite{sta2}]
    Let $V_d$ be a vector space of dimension $d$. Define the vector space $\mathcal{T}^{S^2}_{V_d}[n]=V_d^{\otimes\binom{n}{2}}$ and the graded vector space $\mathcal{T}^{S^2}=\displaystyle\bigoplus_{n\geq0}\mathcal{T}^{S^2}_{V_d}[n]$, where $\mathcal{T}^{S^2}_{V_d}[0]=\mathcal{T}^{S^2}_{V_d}[1]=k$.
\end{definition}
\begin{remark}
    Simple tensors in $\mathcal{T}^{S^2}_{V_d}[n]$ are of the form $(\otimes_{1\leq i<j\leq n}(v_{i,j}))$. In particular, given a simple tensor $\omega=(\otimes_{1\leq i<j\leq n}c_{i,j})$, we can associate a $d\times \binom{n}{2}$ matrix $A_\omega$ by the matrix with columns $c_{i,j}$ ordered by the dictionary order. We will refer to $A_\omega$ as the matrix form of $\omega$ and refer to $\omega$ as the simple tensor associated to $A_\omega$.

    By way of example, let $\omega=(\otimes_{1\leq i<j\leq 4}(v_{i,j}))$, where $$v_{i,j}=\begin{cases}
        e_1, \ \ \ i+j \ even\\
        e_1+e_2, \ \ \ i+j \ odd
    \end{cases}.$$ 
    
    If $A$ is the matrix form of $\omega$, the columns of the matrix $A$ are of the form $(i,j)$ for $1\leq i<j\leq 4$ under the dictionary order (i.e. the columns are labelled by $(1,2), (1,3), (1,4), (2,3), (2,4),$ and $(3,4)$, respectively). In this case, 
    \begin{align*}
        A&=(e_1+e_2)\otimes e_1\otimes (e_1+e_2)\otimes (e_1+e_2)\otimes e_1\otimes (e_1+e_2)\\
        &\\
        &=\begin{array}{cccccccc}
&(1,2)&(1,3)&(1,4)&(2,3)&(2,4)&(3,4)& \\
\ldelim({2}{4mm}&1&1&1&1&1&1&\rdelim){2}{4mm}\\
    &1&0&1&1&0&1&
\end{array}\\
    \end{align*}
\end{remark}

Now, there is the following fact about $\mathcal{T}^{S^2}_{V_d}[n]$.
\begin{proposition}[\cite{edge}]\label{BasisGraph}
    Fix a basis $\mathcal{B}_d=\{e_1,\ldots,e_d\}$ for $V_d$ and define
    \[\mathcal{G}^{S^2}_{\mathcal{B}_d}[n]=\{\otimes_{1\leq i<j\leq n}(v_{i,j})|v_{i,j}\in\mathcal{B}_d\}.\]
    Then, $\mathcal{G}^{S^2}_{\mathcal{B}_d}[n]$ is a system of generators for $\mathcal{T}^{S^2}_{V_d}[n]$. Further, there exists a bijection between $\mathcal{G}^{S^2}_{\mathcal{B}_d}[n]$ and $\mathcal{P}_d(K_n)$, the collection of edge $d$-partitions of the complete graph on $n$ vertices.
\end{proposition}
Proposition \ref{BasisGraph} follows by direct construction. Given $\otimes_{1\leq i<j\leq n}(v_{i,j})\in\mathcal{G}^{S^2}_{\mathcal{B}_d}[n]$, for $1\leq i\leq d$ we define $\Gamma_i$ to be the graph with vertices $1,2,3,\ldots,n$ and edges 
\[E(\Gamma_i)=\{(x,y)|v_{x,y}=e_i\}.\]
This defines an edge $d$-partition and this process defines a bijection.

Next, there is a special element of $\mathcal{T}^{S^2}_{V_d}[2d]$ that we will be relying on in this paper, which generalizes the construction of the simple tensor associated to the identity matrix.
\begin{definition}[\cite{dets2}]
    For $1\leq a\leq d$, let $S_a=\{2a-1,2a\}$. 

    Then, for $1\leq i<j\leq 2d$, we define $e_{i,j}$ by
    \[e_{i,j}=\begin{cases}
        e_a, \ \ \ \ \textrm{$i+j$ is even and $i\in S_a$}\\
        e_b, \ \ \ \ \textrm{$i+j$ is odd and $j\in S_b$}
    \end{cases}.\]
    Finally, we define $E^{(2)}_d\in\mathcal{T}^{S^2}_{V_d}[2d]$ by $E^{(2)}_d=\otimes_{1\leq i<j\leq 2d}(e_{i,j})$.
\end{definition}
\begin{remark}
    Notice that the edge $d$-partition associated with $E^{(2)}_d$ consists only of twin star graphs. More specifically, the edge $d$-partition is $(TS_d(1,2),TS_d(3,4),\ldots,TS_d(2d-1,2d))$.
\end{remark}
\begin{remark}
    In \cite{edge}, a different distinguished element of $\mathcal{T}^{S^2}_{V_d}[2d]$ was given for $d\geq2$. However, the element $E^{(2)}_d$ fits more naturally into a further generalization, so we will be using it for the terms defined in this paper. For more information on this further generalization $E^{(r)}_d$, see \cite{SDiss}.
\end{remark}

\subsection{The Map $det^{S^2}$}
\subsubsection{General Results}
Next, we will consider a map $det^{S^2}:\mathcal{T}^{S^2}_{V_d}[2d]\to k$. This map has been studied in several instances, including \cite{SDiss}, \cite{edge}, \cite{dets2}, \cite{sta2} and \cite{sv}. For our purposes, the following theorem summarizes the relevant information.

\begin{proposition}[\cite{edge},\cite{dets2},\cite{sta2}]\label{dets2thm}
There exists a map $det^{S^2}:\mathcal{T}^{S^2}_{V_d}[2d]\to k$ which is linear, nontrivial, and has the property that $det^{S^2}(\otimes_{1\leq i<j\leq 2d}(v_{i,j}))=0$ if there exists $1\leq x<y<z\leq 2d$ such that $v_{x,y}=v_{x,z}=v_{y,z}$. Further, if $d=2$ or $d=3$, the map $det^{S^2}$ is unique, up to a constant.
\end{proposition}
\begin{remark}
    There are parallels between Propositions \ref{detprop} and \ref{dets2thm}, even if uniqueness is still an open question for $d>4$ in Proposition \ref{dets2thm}. In \cite{SDiss} and \cite{dets3}, it was shown that the determinant map and $det^{S^2}$ have a further generalization $det^{S^r}:V_d^{\otimes\binom{rd}{r}}\to k$ that satisfy a property generalizing that $det^{S^2}(\otimes_{1\leq i<j\leq 2d}(v_{i,j}))=0$ if there exists $1\leq x<y<z\leq 2d$ such that $v_{x,y}=v_{x,z}=v_{y,z}$. We will not be focusing on $det^{S^r}$ in this paper, although many of the concepts introduced can be extended to this generalization.
\end{remark}
Regarding this map $det^{S^2}:V_d^{\otimes d(2d-1)}\to k$, a construction is given in \cite{dets2} for all $d\geq1$. However, there exists an alternative construction for $d=2$ and $d=3$ which is more useful for our purposes in this paper. Due to the uniqueness, the alternative construction and the general construction presented in \cite{dets2} are equivalent for $d=2$ and $d=3$.

\subsubsection{The Case $d=2$}
First, suppose that $d=2$, so $V_d=V_2$ is a $2$-dimensional vector space. Let $v_{i,j}\in V_2$ for $1\leq i<j\leq 4$, where $v_{i,j}=v_{i,j}^1e_1+v_{i,j}^2e_2$.

\begin{comment}
First, we have the following result from \cite{edge}.
\begin{lemma}[\cite{edge}]\label{varepd2}
    \begin{enumerate}
        \item $S_4$ acts on $\mathcal{P}^{h,cf}_d(K_4)$ by permuting the vertices.
        \item $S_2$ acts on $\mathcal{P}^{h,cf}_d(K_4)$ by permuting the order of the partition.
        \item The action by $S_4$ and $S_2$ are disjoint, so there is an action of $S_4\times S_2$ on $\mathcal{P}^{h,cf}_d(K_4)$.
        \item The orbit of $E^{(2)}_2$ under the $S_4\times S_2$ action is $\mathcal{P}^{h,cf}_d(K_4)$.
        \item There exists a map $\varepsilon^{S^2}_2:\mathcal{P}^{h,cf}_d(K_4)\to \{\pm1\}$ such that $\varepsilon^{S^2}_2(E^{(2)}_2)=-1$ and $\varepsilon^{S^2}_2((\sigma,\tau)E^{(2)}_2)=-\varepsilon(\sigma)\varepsilon(\tau)$, where $\varepsilon$ is the sign function of the symmetric group.
    \end{enumerate}
\end{lemma}
\end{comment}

Given an edge $2$-partition $(\Gamma_1,\Gamma_2)$, define the monomial
\begin{equation}
    M_{(\Gamma_1,\Gamma_2)}(\otimes_{1\leq i<j\leq 2d}(v_{i,j}))=\prod_{(i_1,j_1)\in E(\Gamma_1)}v_{i_1,j_1}^1\prod_{(i_2,j_2)\in E(\Gamma_2)}v_{i_2,j_2}^2
\end{equation}

Given this, we have the following result.

\begin{proposition}[\cite{edge}]
    Let $(\Gamma_1^{E^{(2)}_2},\Gamma_2^{E^{(2)}_2})$ be the edge $2$-partition associated with $E^{(2)}_2$. For $1\leq x<y<z\leq 2d$, let $(\Gamma_1,\Gamma_2)^{(x,y,z)}$ be the involution of $(\Gamma_1,\Gamma_2)$ given in Lemma \ref{invol}.
    \begin{enumerate}
    \item There exists a map $\varepsilon^{S^2}_2:\mathcal{P}_2^{h,cf}(K_{4})\to\{\pm1\}$ such that $\varepsilon^{S^2}_2((\Gamma_1,\Gamma_2)^{(x,y,z)})=-\varepsilon^{S^2}_2(\Gamma_1,\Gamma_2)$, $\varepsilon^{S^2}_2$ is unique, and $\varepsilon^{S^2}_2(\Gamma_1^{E^{(2)}_2},\Gamma_2^{E^{(2)}_2})=-1$.
    \item The map determined by
    \begin{equation}\label{dets2d2}
    det^{S^2}(\otimes_{1\leq i<j\leq 2d}v_{i,j})=\sum_{(\Gamma_1,\Gamma_2)\in \mathcal{P}^{h,cf}_2(K_{4})} \varepsilon_2^{S^2}((\Gamma_1,\Gamma_2))M_{(\Gamma_1,\Gamma_2)}((v_{i,j})_{1\leq i<j\leq 4}),
    \end{equation}
    is linear, $det^{S^2}$ is unique up to a constant, and has the property that $det^{S^2}(\otimes_{1\leq i<j\leq 4}v_{i,j})=0$ if there exists $1\leq x<y<z\leq 4$ such that $v_{x,y}=v_{x,z}=v_{y,z}$. 
    \end{enumerate}
\end{proposition}
\begin{remark}
    In \cite{edge}, the map $\varepsilon^{S^2}_2$ was more closely studied, particular in relation to an $S_{4}\times S_2$ action. These facts will not be relevant for our discussion here.
\end{remark}

\subsubsection{The Case $d>2$}
Now, we will use the previous construction to describe a construction for $d=3$, as well as to present a conjecture for $d>3$ proposed in \cite{edge}.

Let $d=3$ and let $v_{i,j}\in V_d$ for $1\leq i<j\leq 2d$, where $v_{i,j}=\sum_{k=1}^dv_{i,j}^ke_k$.

\begin{comment}
\begin{lemma}[\cite{edge},\cite{dim4}]\label{varepd3d4}
    Let $d=2$ or $d=3$.
    \begin{enumerate}
        \item $S_{2d}$ acts on $\mathcal{P}^{h,cf}_d(K_{2d})$ by permuting the vertices.
        \item $S_d$ acts on $\mathcal{P}^{h,cf}_d(K_{2d})$ by permuting the order of the partition.
        \item The action by $S_{2d}$ and $S_d$ are disjoint, so there is an action of $S_{2d}\times S_d$ on $\mathcal{P}^{h,cf}_d(K_{2d})$.
        \item There exists a unique map $\varepsilon^{S^2}_d:\mathcal{P}^{h,cf}_d(K_{2d})\to \{\pm1\}$ such that $\varepsilon^{S^2}_d(E^{(2)}_d)=(-1)^{d+1}$ and $\varepsilon^{S^2}_2((\sigma,\tau)E^{(2)}_2)=(-1)^{d+1}\varepsilon(\sigma)\varepsilon(\tau)$, where $\varepsilon$ is the sign function of the symmetric group.
    \end{enumerate}
\end{lemma}
\begin{remark}
    Notice the parallels between Lemmas \ref{varepd2} and \ref{varepd3d4}. However, the action is not transitive in the case $d>2$. Regardless, we are interested here in the existence of the map $\varepsilon$, not in relation to the group action.
\end{remark}
\end{comment}
First, given an edge $d$-partition $(\Gamma_1,\Gamma_2,\ldots,\Gamma_d)$, define the monomial
\begin{equation}\label{Mond2d3}
    M_{(\Gamma_1,\Gamma_2,\ldots,\Gamma_d)}(\otimes_{1\leq i<j\leq 2d}v_{i,j})=\prod_{k=1}^d\prod_{(i,j)\in E(\Gamma_k)}v_{i,j}^k.
\end{equation}

Then, we have the following proposition from \cite{edge}.
\begin{proposition}[\cite{edge}]\label{dim34prop}
    Let $d=2$ or $d=3$. Let $(\Gamma_1^{E^{(2)}_d},\Gamma_2^{E^{(2)}_d},\ldots,\Gamma_d^{E^{(2)}_d})$ be the edge $d$-partition associated with $E^{(2)}_d$. For $1\leq x<y<z\leq 2d$, let $(\Gamma_1,\Gamma_2,\ldots,\Gamma_d)^{(x,y,z)}$ be the involution of $(\Gamma_1,\Gamma_2,\ldots,\Gamma_d)$ given in Lemma \ref{invol}. 
    \begin{enumerate}
        \item There exists a unique map $\varepsilon^{S^2}_d:\mathcal{P}_d^{h,cf}(K_{2d})\to\{\pm1\}$ such that $\varepsilon^{S^2}_d((\Gamma_1,\Gamma_2,\ldots,\Gamma_d)^{(x,y,z)})=-\varepsilon^{S^2}_d(\Gamma_1,\Gamma_2,\ldots,\Gamma_d)$, $\varepsilon^{S^2}_d$ is unique, and $\varepsilon^{S^2}_d(\Gamma_1^{E^{(2)}_2},\Gamma_2^{E^{(2)}_d})=(-1)^{d+1}$.
        \item The map determined by 
        \begin{equation}\label{dets2d3d4}
    det^{S^2}(\otimes_{1\leq i<j\leq 2d}v_{i,j})=\sum_{(\Gamma_1,\Gamma_2,\ldots,\Gamma_d)\in \mathcal{P}^{h,cf}_d(K_{2d})} \varepsilon_d^{S^2}((\Gamma_1,\Gamma_2,\ldots,\Gamma_d))M_{(\Gamma_1,\Gamma_2,\ldots,\Gamma_d)}(\otimes_{1\leq i<j\leq 2d}v_{i,j}).
    \end{equation}
    is linear, nontrivial, unique up to a constant, and has the property that $det^{S^2}(\otimes_{1\leq i<j\leq 2d}v_{i,j})=0$ if there exists $1\leq x<y<z\leq 2d$ such that $v_{x,y}=v_{x,z}=v_{y,z}$.
    \end{enumerate}
\end{proposition}

Now, the monomial given in Equation \ref{Mond2d3} can be defined for $d>3$. However, Proposition \ref{dim34prop} is not known for $d>3$. With this in mind, there was the following conjecture in \cite{edge} (stated differently in that paper).
\begin{conjecture}\label{dets2Conj}
    Let $d>1$ and $(\Gamma_1^{E^{(2)}_d},\Gamma_2^{E^{(2)}_d},\ldots,\Gamma_d^{E^{(2)}_d})$ be the edge $d$-partition associated with $E^{(2)}_d$. For $1\leq x<y<z\leq 2d$, let $(\Gamma_1,\Gamma_2,\ldots,\Gamma_d)^{(x,y,z)}$ be the involution of $(\Gamma_1,\Gamma_2,\ldots,\Gamma_d)$ given in Lemma \ref{invol}.
    \begin{enumerate}
        \item There exists a map $\varepsilon^{S^2}_d:\mathcal{P}^{h,cf}_d(K_{2d})\to\{\pm1\}$ such that $$\varepsilon^{S^2}_d(E^{(2)}_d)=(-1)^{d+1}$$ and $$\varepsilon^{S^2}_d((\Gamma_1,\Gamma_2,\ldots,\Gamma_d)^{(x,y,z)})=-\varepsilon^{S^2}_d((\Gamma_1,\Gamma_2,\ldots,\Gamma_d))$$ for all $1\leq x<y<z\leq 2d$.
        \item For all $d>1$, the map 
        \[det^{S^2}(\otimes_{1\leq i<j\leq 2d}v_{i,j})=\sum_{(\Gamma_1,\Gamma_2,\ldots,\Gamma_d)\in \mathcal{P}^{h,cf}_d(K_{2d})} \varepsilon_d^{S^2}((\Gamma_1,\Gamma_2,\ldots,\Gamma_d))M_{(\Gamma_1,\Gamma_2,\ldots,\Gamma_d)}(\otimes_{1\leq i<j\leq 2d}v_{i,j}),\]
        is linear, nontrivial, and has the property that $det^{S^2}(\otimes_{1\leq i<j\leq 2d}v_{i,j})=0$ if there exists $1\leq x<y<z\leq 2d$ such that $v_{x,y}=v_{x,z}=v_{y,z}$.
    \end{enumerate}
\end{conjecture}
\begin{remark}
    Although it was not stated in Conjecture \ref{dets2Conj}, it was also conjectured that $\varepsilon^{S^2}_d$ and $det^{S^2}$ are unique up to a constant.
\end{remark}

Before we finish this section, notice that by using the construction behind Theorem \ref{dets2thm}, we can compute $det^{S^2}(E^{(2)}_d)$ for small cases of $d$ (this is a different but equivalent construction than what is given in Propositions \ref{dets2d2} and \ref{dets2d3d4}). More recently, the cases of $d\leq 10$ were computed in the Appendix of \cite{dets3}.
\begin{proposition}[\cite{dets3}]\label{EdProp}
    Let $1\leq d\leq 10$. Then, $det^{S^2}(E^{(2)}_d)=(-1)^{d+1}$.
\end{proposition}
Notice that the result from Proposition \ref{EdProp} matches what would be expected if Conjecture \ref{dets2Conj} holds.
\begin{remark}
    As a last remark for this section, we are relying on Conjecture \ref{dets2Conj} rather than the construction of $det^{S^2}$ for all $d>1$ presented in \cite{dets2} because it will be more convenient for later computations. With that stated, keep in mind that for $d>4$, some of the results presented are dependent on a conjecture which has not been proven for all $d$. Alternatively, one could show that the vector space $\Lambda^{S^2}_{V_d}[2d]$, presented in \cite{sta2}, is generated by the image of $E^{(2)}_d$ for all $d>1$. We will not be discussing the specifics in this paper, but more information regarding this can be seen in \cite{SDiss}.
\end{remark}

\section{$S^2$-Triangular Matrices}
In this section, we introduce a new type of matrix based on Triangular matrices and $E^{(2)}_d$. 
\subsection{$S^2$-Diagonal of a Matrix}
First, we construct a generalization of the concept of the diagonal of a matrix. This definition will parallel Definition \ref{DiagDef}.
\begin{definition}
    Let $A$ be a $d\times d(2d-1)$ matrix with columns $a_{i,j}$ for $1\leq i<j\leq 2d$ ordered by the dictionary order and entries $a_{i,j}^k$ in the $(i,j)$ column and $k^{th}$ row. Let $I^{S^2}_d$ be the matrix form of $E^{(2)}_d$ ordered by the dictionary order with entries $e_{i,j}^k$ in the $(i,j)$ column and $k^{th}$ row.

    Define the $S^2$ diagonal indicator set by $DI_d[2]=\{(i,j,k)|e_{i,j}^k\neq0\}$ and the $S^2$-diagonal of $A$ as the set $diag^{S^2}(A)=\{a_{i,j}^k|(i,j,k)\in DI_d[2]\}$. We will write $DI_d[2]$ as $DI[2]$ if $d$ is understood.
\end{definition}
\begin{remark}
    Notice that by construction of $E^{(2)}_d$, for each $1\leq i<j\leq 2d$, there exists a unique $1\leq k\leq d$ such that $(i,j,k)\in DI[2]$. 
\end{remark}

\begin{example}
Now, we give two examples and their $S^2$-diagonal. For both of these examples, we will take $d=3$, so we will be looking at $3\times15$ matrices.

First, for these examples, we will need $DI[2]$. By using the definition of $E^{(2)}_3$, we have the following matrix for $I^{S^2}_3$, the matrix form of $E^{(2)}_3$. Here the columns here are $(1,2), (1,3), (1,4), (1,5), (1,6), (2,3),$ $(2,4), (2,5), (2,6), (3,4), (3,5), (3,6)$,  $ (4,5), (4,6),$ and $(5,6)$, which we have listed above the respective columns.
\begin{comment}\[I^{S^2}_3=\tiny\begin{blockarray}{ccccccccccccccc}
  (1,2)&(1,3)&(1,4)&(1,5)&(1,6)&(2,3)&(2,4)&(2,5)&(2,6)&(3,4)&(3,5)&(3,6)&(4,5)&(4,6)&(5,6)\\
\begin{block}{[ccccccccccccccc]}
1&1&0&1&0&0&1&0&1&0&0&0&0&0&0\bigstrut[t]\\
    0&0&1&0&0&1&0&0&0&1&1&0&0&1&0\\
    0&0&0&0&1&0&0&1&0&0&0&1&1&0&1\bigstrut[b]\\
\end{block}
\end{blockarray}\]
\end{comment}
\[I^{S^2}_3=\tiny\begin{array}{ccccccccccccccccc}
&(1,2)&(1,3)&(1,4)&(1,5)&(1,6)&(2,3)&(2,4)&(2,5)&(2,6)&(3,4)&(3,5)&(3,6)&(4,5)&(4,6)&(5,6)& \\
\ldelim({3}{4mm}&1&1&0&1&0&0&1&0&1&0&0&0&0&0&0&\rdelim){3}{4mm}\\
    &0&0&1&0&0&1&0&0&0&1&1&0&0&1&0&\\
    &0&0&0&0&1&0&0&1&0&0&0&1&1&0&1&
\end{array}\]
\begin{comment} \begin{matrix}
    \tiny\begin{matrix}
    (1,2)&(1,3)&(1,4)&(1,5)&(1,6)&(2,3)&(2,4)&(2,5)&(2,6)&(3,4)&(3,5)&(3,6)&(4,5)&(4,6)&(5,6)
\end{matrix}\\
    \begin{pmatrix}
    1&1&0&1&0&0&1&0&1&0&0&0&0&0&0\\
    0&0&1&0&0&1&0&0&0&1&1&0&0&1&0\\
    0&0&0&0&1&0&0&1&0&0&0&1&1&0&1
\end{pmatrix}
\end{matrix}
\end{comment}
Since this is the matrix form of $E^{(2)}_3$, the set $DI_3[2]=DI[2]$ is based on the non-zero entries of $I^{S^2}_3$. Thus, we have
\begin{align*}
  DI[2]=&\{(1,2,1), (1,3,1), (1,4,2), (1,5,1), (1,6,3), (2,3,2), (2,4,1), \\
  &(2,5,3), (2,6,1), (3,4,2), (3,5,2), (3,6,3), (4,5,3), (4,6,2), (5,6,3)\}. 
\end{align*}
For example, in the first column (the (1,2) column), we see that the first row is nonzero, so $(1,2,1)\in DI[2]$. However, the other entries are zero so neither $(1,2,2)$ nor $(1,2,3)$ is in $DI[2]$. Similarly, the last column is $(5,6)$ and the only nonzero entry is in the third row, so $(5,6,3)\in DI[2]$.

Next, consider the $S^2$-diagonal $diag^{S^2}(I^{S^2}_3)$. This is determined by looking at all of the positions given in $DI[2]$, ordering all of the columns as $(i,j)$ for $1\leq i<j\leq 6$ using the dictionary order. For example, we know that $(2,5,3)\in DI[2]$, so $e^3_{2,5}$, that is the element at the third row in column $(2,5)$, is in the $S^2$-diagonal. After looking at all of the positions given in $DI[2]$, we get $diag^{S^2}(I^{S^2}_3)=\{1\}$.

Now, consider the example of the following matrix $B=(b_{i,j}^k)_{1\leq i<j\leq 2d,1\leq k\leq d}$:
\[B=\tiny\begin{array}{ccccccccccccccccc}
&(1,2)&(1,3)&(1,4)&(1,5)&(1,6)&(2,3)&(2,4)&(2,5)&(2,6)&(3,4)&(3,5)&(3,6)&(4,5)&(4,6)&(5,6)& \\
\ldelim({3}{4mm}&1&1&0&1&b&c&d&0&1&0&0&0&0&0&m&\rdelim){3}{4mm}\\
    &a&0&0&0&0&1&e&0&0&1&1&0&0&1&0&\\
    &0&0&1&0&1&0&f&1&0&0&0&l&1&0&n&
\end{array}.\]

Notice that not all of the nontrivial entries correspond to indices in $DI[2]$. For example, $b_{1,2}^2=a$, although $(1,2,2)$ is not in $DI[2]$. Keeping this in mind, we get $diag^{S^2}(B)=\{0,1,f,l,n\}$.
\end{example}
\subsection{$S^2$-Triangular Matrices}
Now, we will define our main construction.
\begin{definition}\label{S2UpperDef}
    Let $1\leq i<j\leq 2d$ and define $ind(i,j)$ as the unique $1\leq ind(i,j)\leq d$ such that $(i,j,ind(i,j))\in DI[2]$.

    Let $A$ be a $d\times d(2d-1)$ matrix with columns $c_{i,j}$ for $1\leq i<j\leq d$ ordered by the dictionary order. Then, we say $A$ is $S^2$-Upper Triangular if for each $1\leq i<j\leq 2d$ and $1\leq k\leq d$, there exists $\alpha_{i,j}^k$ such that
    \[c_{i,j}=\sum_{k=1}^{ind(i,j)}\alpha_{i,j}^ke_k.\]
    We say that $A$ is $S^2$-Lower Triangular if for each $1\leq i<j\leq 2d$ and $1\leq k\leq d$, there exists $\alpha^k_{i,j}$ such that 
    \[c_{i,j}=\sum_{k=ind(i,j)}^d\alpha_{i,j}^ke_k.\]
    If $A$ is either $S^2$-Upper Triangular or $S^2$-Lower Triangular, we say that $A$ is an $S^2$-Triangular Matrix.
\end{definition}

\begin{remark}
    Note that $I^{S^2}_d$ is both an $S^2$-Upper Triangular Matrix and an $S^2$-Lower Triangular Matrix, taking $\alpha_{i,j}^k=\delta_{k,ind(i,j)}$, where $\delta_{m,n}$ is the Kronecker delta.
\end{remark}
\begin{remark}
    There are noticeable parallels between the reformulation of the definition of Triangular Matrices in Definition \ref{UpperLowerDef} and $S^2$-Triangular Matrices in Definition \ref{S2UpperDef}. This can be extended to $S^r$-Upper Triangular Matrices by using $E^{(r)}_d$, defined in \cite{dets3}.
\end{remark}
Next, we give some examples.
\begin{example}
    First, consider the following matrix $A$, where the indices in $DI[2]$ are boxed.
    \[A=\begin{pmatrix}
    \boxed{a}&\boxed{b}&c&\boxed{e}&0&f&\boxed{1}&h&\boxed{1}&0&l&0&0&0&0\\
    0&0&\boxed{0}&0&0&\boxed{g}&0&0&0&\boxed{k}&\boxed{1}&m&n&\boxed{1}&p\\
    0&0&0&0&\boxed{0}&0&0&\boxed{1}&0&0&0&\boxed{1}&\boxed{1}&0&\boxed{1}
    \end{pmatrix}\]
    This matrix is an $S^2$-Upper Triangular Matrix, as all non-zero entries lie at or above the boxed entries in any given column.

    Next, consider the following matrix $B$, where the indices in $DI[2]$ are boxed.
    \[B=\begin{pmatrix}
    \boxed{a}&\boxed{b}&c&\boxed{e}&0&f&\boxed{1}&h&\boxed{1}&0&l&0&0&0&0\\
    0&0&\boxed{0}&0&0&\boxed{g}&0&0&0&\boxed{k}&\boxed{1}&m&n&\boxed{1}&p\\
    0&z&0&0&\boxed{0}&0&0&\boxed{1}&0&0&0&\boxed{1}&\boxed{1}&0&\boxed{1}
    \end{pmatrix}\]
    Notice that the entry at column $(1,3)$ and row $3$ is nonzero, although $ind(1,3)=1$. Therefore, $B$ is not an $S^2$-Upper Triangular Matrix. Similarly, $B$ is not an $S^2$-Lower Triangular Matrix since there is a nonzero entry in column $(1,4)$ and row $1$, although $ind(1,4)=2$.
\end{example}
\begin{remark}
    With a $d\times d$ Upper Triangular Matrix $A$, it is well known that if $Ax=b$ and the diagonal entries of $A$ are non-zero, then we can use back substitution to determine $x$. 
    
    In the case of an $S^2$-Upper Triangular Matrix $A$, we can consider the system $Ax=b$, where $x=(x^{(1,2)},x^{(1,3)},x^{(1,2d)},x^{(2,3)},\ldots,x^{(2d-1,2d)})\in V_{d(2d-1)}$ is labelled by the dictionary order and $b=(b^1,b^2,\ldots,b^d)\in V_d$. In this case, we know that certain $x^{(i,j)}$ are linear combinations of each other, but cannot determine $x^{(i,j)}$ exactly. For example, if we consider the matrix $A$ given in the previous example, then
    \[x^{(2,5)}+x^{(3,6)}+x^{(4,5)}+x^{(5,6)}=b^3,\]
    so we could write $x^{(2,5)}$ as a linear combination of $x^{(3,6)}, x^{(4,5)}, $ and $x^{(5,6)}$, but we cannot completely determine $x$. Notice that $(2,5)$, $(3,6)$, $(4,5)$, and $(5,6)$ are all edges of $\Gamma_3$, where $(\Gamma_1,\Gamma_2,\Gamma_3)$ is the edge $3$-partition associated with $E^{(2)}_3$. 
    
    However, note that if $A$ is $S^2$-Upper Triangular and $0$ is not in the $S^2$-diagonal of $A$, then it follows that we can write $x^{(2a-1,2a)}$ as a linear combination for $1\leq a\leq d$. Namely, for each $1\leq a\leq d$, there exists $\alpha^{(2a-1,2a)}_{m,n}\in k$ for $1\leq m<n\leq 2d$ such that
    \[x^{(2a-1,2a)}=\sum_{\substack{(m,n)\in E(\Gamma_k), \ k\geq a\\ (m,n)\neq(2b-1,2b), \ b\geq a}}\alpha^{(2a-1,2a)}_{m,n}x^{(m,n)}.\]
\end{remark}
\subsection{The Determinant of $S^2$-Upper Triangular Matrices}
We will end this section by proving a result which parallels Proposition \ref{detprop}. First, we start with some lemmas.

\begin{lemma}\label{lem1}
    Let $\Gamma^{E^{(2)}_d}=(\Gamma_1,\Gamma_2,\ldots,\Gamma_d)$ be the edge $d$-partition associated with $E^{(2)}_d$ and let $A$ be a $d\times d(2d-1)$ matrix with associated simple tensor $\otimes_{1\leq i<j\leq 2d}(v_{i,j})$. Then, \begin{equation}\label{lem1eq}M_{\Gamma^{E_d^{(2)}}}(\otimes_{1\leq i<j\leq 2d}(v_{i,j}))=prod(diag^{S^2}(A)).\end{equation}
\end{lemma}
\begin{proof}
    For ease of notation, we will label $\Gamma^{E^{(2)}_d}=\Gamma$.

    Notice that both terms appearing in Equation \ref{lem1eq} are monomials. So, we will proceed by showing that they consists of products of the same terms.

    First, let $\otimes_{1\leq i\leq j\leq 2d}(v_{i,j})$ be the simple tensor associated with $A$ and suppose that $v_{i,j}^k$ is a term in $M_{\Gamma}$. Then, by definition of $M_{\Gamma}$, we know that the edge $(i,j)$ is in the subgraph $\Gamma_k$ of the complete graph. However, this also means that $(i,j,k)\in DI[2]$ since $\Gamma_k$ is from the edge $d$-partition associated with $E^{(2)}_d$, so $v_{i,j}^k\in diag^{S^2}(A)$. In particular, the terms appearing in $M_{\Gamma}$ also appears in $prod(diag^{S^2}(A))$.

    Next, suppose that $v_{x,y}^z\in diag^{S^2}(A)$. Then, we know that the edge $(x,y)$ lies on the graph $\Gamma_z$ by definition of $diag^{S^2}(A)$ since $\Gamma=\Gamma^{E^{(2)}_d}$. In particular, this implies that $v_{x,y}^z$ is also a term in the product $M_{\Gamma}$. In particular, all of the terms of $prod(diag^{S^2}(A))$ also appear in $M_{\Gamma}$. Therefore, Equation \ref{lem1eq} holds.
\end{proof}

\begin{lemma}\label{lem2}
    If $A$ is an $S^2$-Upper Triangular Matrix with associated simple tensor $\otimes_{1\leq i<j\leq 2d}(v_{i,j})$ and $\Gamma=(\Gamma_1,\Gamma_2,\ldots,\Gamma_d)$ is a homogeneous cycle-free edge $d$-partition of the complete graph $K_{2d}$, then either $\Gamma$ is the edge $d$-partition associated with $E^{(2)}_d$ or $M_\Gamma(\otimes_{1\leq i<j\leq 2d}(v_{i,j}))=0$.
\end{lemma}
\begin{proof}
    Let $\otimes_{1\leq i<j\leq 2d}(v_{i,j})$ be the simple tensor associated with $A$ and let $\Gamma^{E^{(2)}_d}=(\Gamma^{E^{(2)}_d}_1,\Gamma^{E^{(2)}_d}_2,\ldots,\Gamma^{E^{(2)}_d}_d)$ be the edge $d$-partition associated with $E^{(2)}_d$.

    Consider the following set:
    \[S=\{(a,b,c)|(a,b)\in E(\Gamma_c)\textrm{, but }(a,b)\notin E(\Gamma^{E^{(2)}_d}_c)\}.\]
    Since both $\Gamma$ and $\Gamma^{E^{(2)}_d}$ are edge $d$-partitions of $K_{2d}$, we know that all of their subgraphs have the same number of edges. Now, $S$ is a finite set so the maximum of $\{c|(a,b,c)\in S\}$ exists. Also, in order for the edge $(a,b)$ to be in $\Gamma_c$ for the maximal $c$ value but not in $\Gamma^{E^{(2)}_d}_c$, that implies that there is some $(a_0,b_0,c_0)$ with $c_0<c$ such that $(a_0,b_0)\notin E(\Gamma_{c_0})\textrm{ and }(a_0,b_0)\in E(\Gamma^{E^{(2)}_d}_{c_0})$. But, by the definition of $S^2$-Upper Triangular Matrices, we know that 
    \[v_{i,j}=\sum_{k=1}^{ind(i,j)}v_{i,j}^k.\]
    Further, $c_0=ind(a_0,b_0)<c$ and so $v_{a_0,b_0}^c=0$. This means that 
    \[M_\Gamma(\otimes_{1\leq i<j\leq 2d}(v_{i,j}))=\prod_{k=1}^d\prod_{(i,j)\in E(\Gamma_k)} v_{i,j}^k=0.\]
\end{proof}
\begin{remark}
    Note that just because $A$ is an $S^2$-Upper Triangular Matrix does not necessarily imply that $M_{\Gamma^{E^{(2)}_d}}(\otimes_{1\leq i<j\leq 2d}(v_{i,j}))\neq0$.
\end{remark}

Next, we give a result paralleling Proposition \ref{detprop2}.
\begin{theorem}\label{mainthm}
    Let $A$ be an $S^2$-Upper Triangular Matrix over $V_d$, let $\otimes_{1\leq i<j\leq 2d}v_{i,j}$ be the simple tensor associated with $A$, and suppose that Conjecture \ref{dets2Conj} holds. Let the product of the $S^2$-diagonal of $A$ be denoted by $prod(diag^{S^2}(A))$. Then, 
    \[det^{S^2}(A)=det^{S^2}(E^{(2)}_d)\cdot prod(diag^{S^2}(A)).\]
\end{theorem}
\begin{proof}
    First, notice that since $E^{(2)}_d$ is $S^2$-Upper Triangular, it follows by Lemma \ref{lem2} that
    \begin{align*}
        det^{S^2}(E^{(2)}_d)&=\sum_{\Gamma\in\mathcal{P}^{h,cf}_d(K_{2d})}\varepsilon^{S^2}_d(\Gamma)M_{\Gamma}(\otimes_{1\leq i<j\leq 2d}(e_{i,j}))\\
        &=\varepsilon^{S^2}_d(\Gamma^{E_d^{(2)}})M_{\Gamma^{E_d^{(2)}}}(\otimes_{1\leq i<j\leq 2d}(e_{i,j})).
    \end{align*}
    Further, by Lemma \ref{lem1}, we know that $M_{\Gamma^{E^{(2)}}}(\otimes_{1\leq i<j\leq 2d}(e_{i,j}))=prod(diag^{S^2}(E_d^{(2)})$. So, since $DI[2]$ identifies all of the nonzero entries of $E^{(2)}_d$, which are all 1 by construction, we know that $det^{S^2}(E^{(2)}_d)=\varepsilon^{S^2}_d(\Gamma^{E^{(2)}})$.

    Next, by Lemmas \ref{lem1} and \ref{lem2},
    \begin{align*}
        det^{S^2}(A)&=\sum_{\Gamma\in\mathcal{P}^{h,cf}_d(K_{2d})}\varepsilon^{S^2}_d(\Gamma)M_{\Gamma}(\otimes_{1\leq i<j\leq 2d}(v_{i,j}))\\
        &=\varepsilon^{S^2}_d(\Gamma^{E^{(2)}})M_{\Gamma^{E^{(2)}}}(\otimes_{1\leq i<j\leq 2d}(v_{i,j}))\\
        &=det^{S^2}(E^{(2)}_d)prod(diag^{S^2}(A)).
    \end{align*}
    since $A$ is an $S^2$-Upper Triangular Matrix.
\end{proof}
\begin{remark}
    Note that similar proofs for Lemma \ref{lem2} and Theorem \ref{mainthm} could be used for $S^2$-Lower Triangular Matrices.
\end{remark}
\begin{remark}
    Theorem \ref{mainthm} can be quickly shown using GNU Octave for $d\leq5$, which does not rely on proving Conjecture \ref{dets2Conj}, since we know $det^{S^2}(E^{(2)}_d)$ for $d\leq10$ by Proposition \ref{EdProp}. Due to the constraints of the symbolic package, we stopped the computations beyond this point. 
\end{remark}
Lastly, we have the following corollary, which follows since we know Conjecture \ref{dets2Conj} is true for $d=2$ or $d=3$.
\begin{corollary}\label{ProdCorr}
     Suppose that $d=2$ or $d=3$. Let $A$ be an $S^2$-Upper Triangular Matrix over $V_d$, let $\otimes_{1\leq i<j\leq 2d}(v_{i,j})$ be the simple tensor associated with $A$, and let the product of the $S^2$-diagonal of $A$ be denoted by $prod(diag^{S^2}(A))$. Then, 
    \[det^{S^2}(A)=det^{S^2}(E_d)\cdot prod(diag^{S^2}(A)).\]
\end{corollary}
\begin{remark}
    Conjecture \ref{dets2Conj} is of current research interest, as seen in \cite{SDiss} and \cite{edge}. Notice that if Conjecture \ref{dets2Conj} can be shown for $d>3$, then Corollary \ref{ProdCorr} will hold as well for $d>3$ by Theorem \ref{mainthm}.
\end{remark}
\begin{landscape}
\section{The Algebra of $S^2$-Upper Triangular Matrices}
In this section, we will construct a $k$-algebra using $S^2$-Upper Triangular Matrices. 
\subsection{The Algebra of $d\times d(2d-1)$ Matrices and the Subalgebra $U^{S^2}_d$}
\subsubsection{The Case $d=2$}First, suppose that $d=2$. Recall that given $m,n\geq1$, the collection of $m\times n$ matrices forms a vector space. With that in mind, we give the following definition.
\begin{definition}
    Let $Mat^{S^2}_2$ be the vector space of $2\times 6$ matrices. Define a map $\cdot:Mat^{S^2}_2\times Mat^{S^2}_2\to Mat^{S^2}_2$ by
    \begin{align}\label{S2Mult}
        &\begin{pmatrix}
            a_{1,2}&a_{1,3}&a_{1,4}&a_{2,3}&a_{2,4}&a_{3,4}\\
            b_{1,2}&b_{1,3}&b_{1,4}&b_{2,3}&b_{2,4}&b_{3,4}
        \end{pmatrix}\cdot \begin{pmatrix}
            c_{1,2}&c_{1,3}&c_{1,4}&c_{2,3}&c_{2,4}&c_{3,4}\\
            d_{1,2}&d_{1,3}&d_{1,4}&d_{2,3}&d_{2,4}&d_{3,4}
        \end{pmatrix}\\
        &=\small\begin{pmatrix}
            a_{1,2}c_{1,2}+a_{3,4}d_{1,2}&a_{1,3}c_{1,3}+a_{2,3}d_{1,3}&a_{2,4}c_{1,4}+a_{1,4}d_{1,4}&a_{1,3}c_{2,3}+a_{2,3}d_{2,3}&a_{2,4}c_{2,4}+a_{1,4}d_{2,4}&a_{1,2}c_{3,4}+a_{3,4}d_{3,4}\\
            b_{1,2}c_{1,2}+b_{3,4}d_{1,2}&b_{1,3}c_{1,3}+b_{2,3}d_{1,3}&b_{2,4}c_{1,4}+b_{1,4}d_{1,4}&b_{1,3}c_{2,3}+b_{2,3}d_{2,3}&b_{2,4}c_{2,4}+b_{1,4}d_{2,4}&b_{1,2}c_{3,4}+b_{3,4}d_{3,4}
        \end{pmatrix}\nonumber
    \end{align}
\end{definition}
\begin{remark}
    The multiplication given by Equation \ref{S2Mult} generalizes multiplication of $d\times d$ matrices. Namely, if we group columns $(1,2)$ with $(3,4)$, $(1,3)$ with $(2,3)$, and $(2,4)$ with $(1,4)$, Equation \ref{S2Mult} is computed to matrix multiplication on these groupings. 
\end{remark}
\begin{lemma}\label{MatAlgLem}
    $Mat^{S^2}_2$ is a $k$-algebra with multiplication given by Equation \ref{S2Mult}.
\end{lemma}
The proof to Lemma \ref{MatAlgLem} follows straightforwardly by verifying the associativity and distributively of Equation \ref{S2Mult}, as we know $Mat^{S^2}_2$ is a vector space. Now, another interesting question is about the unit of $Mat^{S^2}_2$. Since the matrix associated to $E^{(2)}_2$ fits into a generalization of the identity matrix, this appears to be a good candidate, which is established in the following lemma.
\begin{lemma}\label{dim2Unit}
    The multiplicative unit of $Mat^{S^2}_2$ is the matrix associated to $E^{(2)}_2$.
\end{lemma}
\begin{proof}
    Let $I^{S^2}_2$ be the matrix associated to $E^{(2)}_2$. We will show that $I^{S^2}_2$ is a multiplicative identity with respect to left multiplication, as the computation for right multiplication is similar. 
    \begin{align*}
        &\begin{pmatrix}
            1&1&0&0&1&0\\
            0&0&1&1&0&1
        \end{pmatrix}\cdot \begin{pmatrix}
            a_{1,2}&a_{1,3}&a_{1,4}&a_{2,3}&a_{2,4}&a_{3,4}\\
            b_{1,2}&b_{1,3}&b_{1,4}&b_{2,3}&b_{2,4}&b_{3,4}
        \end{pmatrix}\\
        &=\small\begin{pmatrix}
            1\cdot a_{1,2}+0\cdot b_{1,2}&1\cdot a_{1,3}+0\cdot b_{1,3}&1\cdot a_{1,4}+0\cdot b_{1,4}&1\cdot a_{2,3}+0\cdot b_{2,3}&1\cdot a_{2,4}+0\cdot b_{2,4}&1\cdot a_{3,4}+0\cdot b_{3,4}\\
            0\cdot a_{1,2}+1\cdot b_{1,2}&0\cdot a_{1,3}+1\cdot b_{1,3}&0\cdot a_{1,4}+1\cdot b_{1,4}&0\cdot a_{2,3}+1\cdot b_{2,3}&0\cdot a_{2,4}+1\cdot b_{2,4}&0\cdot a_{3,4}+1\cdot b_{3,4}
        \end{pmatrix}\\
        &=\begin{pmatrix}
            a_{1,2}&a_{1,3}&a_{1,4}&a_{2,3}&a_{2,4}&a_{3,4}\\
            b_{1,2}&b_{1,3}&b_{1,4}&b_{2,3}&b_{2,4}&b_{3,4}
        \end{pmatrix}
    \end{align*}
\end{proof}
\end{landscape}
\begin{proposition}\label{Upperdim2}
    Let $U^{S^2}_2$ be the collection of $2\times 6$ $S^2$-Upper Triangular Matrices. Then, $U^{S^2}_2$ is a subalgebra of $Mat^{S^2}_2$.
\end{proposition}
\begin{proof}
    It suffices to show that $U^{S^2}_2$ is closed under multiplication, which comes from direct computation, as the $U^{S^2}_2$ is a vector subspace of $Mat^{S^2}_2$. We box the $S^2$-diagonal of the resulting matrix to verify it is an $S^2$-Upper Triangular Matrix.

    \begin{align*}
        \begin{pmatrix}
            a&b&c&e&g&h\\
            0&0&d&f&0&i
        \end{pmatrix}&\cdot\begin{pmatrix}
            j&k&m&p&r&s\\
            0&0&n&q&0&t
        \end{pmatrix}\\
        &=\begin{pmatrix}
            \boxed{aj}&\boxed{bk}&gm+cd&bp+eq&\boxed{gr}&as+ht\\
            0&0&\boxed{dn}&\boxed{fq}&0&\boxed{it}
        \end{pmatrix}
    \end{align*}
\end{proof}
\subsubsection{The Case $d>2$}
Now, we will construct $U_d^{S^2}$ for $d>2$.
\begin{definition}
    Let A be a $d\times d(2d-1)$ matrix associated with the simple tensor $\otimes_{1\leq i<j\leq 2d}(c_{i,j})$. Let $1\leq n\leq 2d-2$ and let $(\Gamma_1,\Gamma_2,\ldots,\Gamma_d)$ be the edge $d$-partition associated with $E^{(2)}_d$.

    For $1\leq a\leq d$, let $u_a$ be the vector in $\{c_{i,j}|1\leq i<j\leq 2d\}$ such that $(i,j)$ is a left leg of the twin star graph $TS_d(2a-1,2a)=\Gamma_a$ with leg number $n$. Take $A^L_n$ be the matrix associated with $u_1\otimes u_2\otimes\ldots\otimes u_d$. Similarly, for $1\leq b\leq d$, let $v_b$ be the vector in $\{c_{i,j}|1\leq i<j\leq 2d\}$ such that $(i,j)$ is a right leg of the graph $TS_d(2b-1,2b)=\Gamma_b$ with leg number $n$. Take $A^R_n$ to be the matrix associated with $v_1\otimes v_2\otimes\ldots\otimes v_d$. Lastly, let $A^C$ be the matrix associated with $c_{1,2}\otimes c_{3,4}\otimes c_{5,6}\otimes\ldots\otimes c_{2d-1,2d}$.

    We call $A^L_n$, $A^R_n$, and $A^C$ the leg submatrices of $A$.
\end{definition}
\begin{remark}
    Notice that the leg submatrices of $A$ are $d\times d$ matrices. Further, the leg submatrices of $A$ are triangular if and only if $A$ is $S^2$-Triangular.
\end{remark}
\begin{example}
    Now, we will give an example of some leg submatrices. First, take the matrix $B$ previously introduced as
    \[B=\tiny\begin{array}{ccccccccccccccccc}
&(1,2)&(1,3)&(1,4)&(1,5)&(1,6)&(2,3)&(2,4)&(2,5)&(2,6)&(3,4)&(3,5)&(3,6)&(4,5)&(4,6)&(5,6)& \\
\ldelim({3}{4mm}&1&1&0&1&b&c&d&0&1&0&0&0&0&0&m&\rdelim){3}{4mm}\\
    &a&0&0&0&0&1&e&0&0&1&1&0&0&1&0&\\
    &0&0&1&0&1&0&f&1&0&0&0&l&1&0&n&
\end{array}.\]
First, consider $B^C$. This is the submatrix given by columns $(1,2)$, $(3,4)$, and $(5,6)$.
\[B^C=\begin{pmatrix}
    1&0&m\\
    \alpha&1&0\\
    0&0&n
\end{pmatrix}\]
Next, consider $B^R_2$. For this, we want to find the edges with leg number two that belong to $R_3(2a-1,2a)$ for some $1\leq a\leq d$. For example, we know that $R_3(1,2)=\{(2,4),(2,6)\}$, $R_3(3,4)=\{(1,4),(4,6)\}$, and $R_3(5,6)=\{(1,6),(3,6)\}$. So, $A^R_2$ consists of the columns of $B$ labelled $(2,6)$, $(4,6)$, and $(3,6)$.
\[B^R_2=\begin{pmatrix}
    1&0&0\\
    0&1&0\\
    0&0&l
\end{pmatrix}\]
Lastly, consider $B^L_1$. For this, we want to find the edges with leg number one that belong to $L_3(2a-1,2a)$ for some $1\leq a\leq d$. We know that $L_3(1,2)=\{(1,3),(1,5)\}$, $L_3(3,4)=\{(2,3),(3,5)\}$, and $L_3(5,6)=\{(2,5),(4,5)\}$. So, $B^L_1$ consists of the columns of $A$ labelled $(1,3)$, $(2,3)$, and $(2,5)$.
\[B^L_1=\begin{pmatrix}
    1&e&0\\
    0&1&0\\
    0&0&1\\
\end{pmatrix}\]
\end{example}
Next, we give a new operation that generalizes what was previously constructed in Equation \ref{S2Mult}.
\begin{definition}\label{LIMDef}
    Let $(\Gamma_1,\Gamma_2,\ldots,\Gamma_d)$ be the edge $d$-partition associated with $E^{(2)}_d$.
    
    Let $A$ be a $d\times d(2d-1)$ matrix with associated simple tensor $\otimes_{1\leq i<j\leq 2d}(a_{i,j})$. Let $A^L_i$, $A^R_i$, and $A^C$ be the leg submatrices for $1\leq i\leq 2d-2$. %with associated simple tensors $l_{1,i}\otimes l_{2,i}\otimes\ldots\otimes l_{d,i}$, $r_{1,i}\otimes r_{2,i}\otimes\ldots\otimes r_{d,i}$, and $c_1\otimes c_2\otimes\ldots\otimes c_d$, respectively.
    
    Similarly, let $B$ be a $d\times d(2d-1)$ matrix with associated simple tensor $\otimes_{1\leq i<j\leq 2d}(b_{i,j})$. Let $B^L_i$, $B^R_i$, and $B^C$ be the leg submatrices for $1\leq i\leq 2d-2$. % with associated simple tensors $l'_{1,i}\otimes l'_{2,i}\otimes\ldots\otimes l'_{d,i}$, $r'_{1,i}\otimes r'_{2,i}\otimes\ldots\otimes r'_{d,i}$, and $c'_1\otimes c'_2\otimes\ldots\otimes c'_d$.

    Define the matrix $A\odot B$ with associated simple tensor $\otimes_{1\leq i<j\leq 2d}(v_{i,j})$ as follows. Let $1\leq m<n\leq 2d$ with $(m,n)$ an edge in $\Gamma_a$ for $1\leq a\leq d$ and leg number $l_0$. If $(m,n)$ is a left leg of $\Gamma_a$, then $v_{i,j}$ is given by the $a^{\textrm{th}}$ column of the matrix $A^L_{l_0}\cdot B^L_{l_0}$, where the multiplication here is the usual matrix multiplication. If $(m,n)$ is a right leg of $\Gamma_a$, then $v_{i,j}$ is given by the $a^{\textrm{th}}$ column of the matrix $A^R_{l_0}\cdot B^R_{l_0}$, where the multiplication is the usual matrix multiplication. Lastly, if $(m,n)=(2a-1,2a)$, then $v_{i,j}$ is given by the $a^{\textrm{th}}$ column of the matrix $A^C\cdot B^C$, where the multiplication is the usual matrix multiplication. We will call the operation $\odot$ the Leg Identifying Multiplication (LIM) of $A$ and $B$.
\end{definition}
\begin{remark}
    Leg Identifying Multiplication is the terminology used here as it involves grouping together columns with the same leg number and are either all in $L_d$, $R_d$, or $C_d$. 
\end{remark}
\begin{remark}
    Here we will maintain a distinction between $\cdot$ for the usual matrix multiplication and $\odot$ for LIM.
\end{remark}
\begin{landscape}
\begin{example}
    Now, we will give two examples. First, we will multiply $I^{S^2}_3\odot B$ using LIM, where
    \[B=\tiny\begin{array}{ccccccccccccccccc}
&(1,2)&(1,3)&(1,4)&(1,5)&(1,6)&(2,3)&(2,4)&(2,5)&(2,6)&(3,4)&(3,5)&(3,6)&(4,5)&(4,6)&(5,6)& \\
\ldelim({3}{4mm}&1&1&0&1&b&c&d&0&1&0&0&0&0&0&m&\rdelim){3}{4mm}\\
    &a&0&0&0&0&1&e&0&0&1&1&0&0&1&0&\\
    &0&0&1&0&1&0&f&1&0&0&0&l&1&0&n&
\end{array}.\]
Notice that the leg submatrix $(I^{S^2}_3)^C$ is the identity matrix $I_3$. In particular, $I_3\cdot B^C=B^C$. This means that since we have the vector $\begin{pmatrix}
        1\\
        \alpha\\
        0
    \end{pmatrix}$ in the first column of $(I^{S^2}_3)^C\cdot B^C$, which is labelled $(1,2)$, the $(1,2)$ column of $I^{S^2}_3\odot B$ is given by $\begin{pmatrix}
        1\\
        \alpha\\
        0
    \end{pmatrix}$. We could continue this computation to verify $I^{S^2}_3\odot B=B$.

    Next, we will multiply $A\otimes B$, where $A$ is given by 
    \[A=\tiny\begin{array}{ccccccccccccccccc}
&(1,2)&(1,3)&(1,4)&(1,5)&(1,6)&(2,3)&(2,4)&(2,5)&(2,6)&(3,4)&(3,5)&(3,6)&(4,5)&(4,6)&(5,6)& \\
\ldelim({3}{4mm}&a&b&c&e&0&f&g&0&1&h&0&0&0&0&l&\rdelim){3}{4mm}\\
    &a&0&m&0&0&1&q&0&e&1&1&0&k&1&0&\\
    &0&0&1&0&1&0&f&1&0&0&0&l&1&z&n&
\end{array}.\]
    First, notice that
    \[A^C=\begin{pmatrix}
        a&h&l\\
        a&1&0\\
        0&0&n
    \end{pmatrix},\]
    and
    \[B^C=\begin{pmatrix}
        1&0&m\\
        a&1&0\\
        0&0&n
    \end{pmatrix}.\]
    Therefore,
    \[A^C\cdot B^C=\begin{pmatrix}
        a+ah&h&am+ln\\
        a^2+a&ah+1&al\\
        0&0&n^2
    \end{pmatrix}.\]
    Now, we can use $A^C\cdot B^C$ for columns $(1,2), (3,4),$ and $(5,6)$ of $A\odot B$
    \[A\odot B=\tiny\begin{array}{ccccccccccccccccc}
&(1,2)&(1,3)&(1,4)&(1,5)&(1,6)&(2,3)&(2,4)&(2,5)&(2,6)&(3,4)&(3,5)&(3,6)&(4,5)&(4,6)&(5,6)& \\
\ldelim({3}{4mm}&a+h&&&&&&&&&h&&&&&am+ln&\rdelim){3}{4mm}\\
    &a^2+a&&&&&&&&&ah+1&&&&&al&\\
    &0&&&&&&&&&0&&&&&n^2&
\end{array}.\]
    Similarly, we can consider $A^L_1$ and $B^L_1$. Remember that the first left legs are the edges $(1,3)$, $(2,3)$, and $(2,5)$. So, using these columns, we get 
    \[A^L_1=\begin{pmatrix}
        b&f&0\\
        0&1&0\\
        0&0&1
    \end{pmatrix}\]
    and
    \[B^L_1=\begin{pmatrix}
        1&c&0\\
        0&1&0\\
        0&0&1
    \end{pmatrix}.\]
    Thus,
    \[A^L_1\cdot B^L_1=\begin{pmatrix}
        b&bc+f&0\\
        0&1&0\\
        0&0&1
    \end{pmatrix}\]
    which allows us to fill in positions $(1,3)$, $(2,3)$, and $(2,5)$ of $A\odot B$.
    \[A\odot B=\tiny\begin{array}{ccccccccccccccccc}
&(1,2)&(1,3)&(1,4)&(1,5)&(1,6)&(2,3)&(2,4)&(2,5)&(2,6)&(3,4)&(3,5)&(3,6)&(4,5)&(4,6)&(5,6)& \\
\ldelim({3}{4mm}&a+h&b&&&&bc+f&&0&&h&&&&&am+ln&\rdelim){3}{4mm}\\
    &a^2+a&0&&&&1&&0&&ah+1&&&&&al&\\
    &0&0&&&&0&&1&&0&&&&&n^2&
\end{array}.\]
    Similarly, the other leg submatrices can be used to obtain the other entries in $A\odot B$.
   \[A\odot B=\tiny\begin{array}{ccccccccccccccccc}
&(1,2)&(1,3)&(1,4)&(1,5)&(1,6)&(2,3)&(2,4)&(2,5)&(2,6)&(3,4)&(3,5)&(3,6)&(4,5)&(4,6)&(5,6)& \\
\ldelim({3}{4mm}&a+h&b&0&e&bg&bc+f&gd+ce&0&1&h&0&0&0&0&am+ln&\rdelim){3}{4mm}\\
    &a^2+a&0&0&0&bq&1&dq+em&0&e&ah+1&1&0&k&1&al&\\
    &0&0&1&0&bf+1&0&df+e+f&1&0&0&0&l^2&1&z&n^2&
\end{array}.\]
\end{example}
Now, notice that it follows directly that we have the following parallels to Lemmas \ref{MatAlgLem} and \ref{dim2Unit}.
\begin{lemma}\label{MatAlgLemmaFull}
    Let $Mat^{S^2}_d$ be the collection of $d\times d(2d-1)$ matrices. Then, $Mat^{S^2}_d$ is a k-algebra with multiplication given by LIM (Definition \ref{LIMDef}) and the matrix associated with $E^{(2)}_d$ is the multiplicative identity.
\end{lemma}
Also, we can extend Proposition \ref{Upperdim2} as follows.
\begin{proposition}\label{PropUS2}
    Let $U^{S^2}_d$ be the collection of $d\times d(2d-1)$ $S^2$-Upper Triangular Matrices. Then, $U^{S^2}_d$ is a subalgebra of $Mat^{S^2}_d$.
\end{proposition}
Notice that Proposition \ref{PropUS2} is easily verified directly, as it can be verified that the leg submatrices of $S^2$-Upper Triangular Matrices are Upper Triangular Matrices and Upper Triangular Matrices are closed under matrix multiplication.
\end{landscape}
\subsection{$S^2$-LU-Decomposition}
Lastly in this section, we will develop a parallel result to Proposition \ref{LUProp}.
\begin{proposition}\label{S2LU}
    Let $A$ be a $d\times d(2d-1)$ matrix and suppose that all of the leg submatrices have nonzero leading principal minors. Then, there exists an $S^2$-Lower Triangular Matrix $L$ and an $S^2$-Upper Triangular Matrix $U$ such that $A=LU$.
\end{proposition}
\begin{proof}
    This proposition follows immediately from the definition of the LIM and construction of leg submatrices. We will show this explicitly for $d=2$.
    
    Let $A$ be given by
    \[A=\begin{pmatrix}
        a_{1,2}&a_{1,3}&a_{1,4}&a_{2,3}&a_{2,4}&a_{3,4}\\
        b_{1,2}&b_{1,3}&b_{1,4}&b_{2,3}&b_{2,4}&b_{3,4}
    \end{pmatrix}.\]
    Then, we can construct matrices $L$ and $U$ by
    \[L=\begin{pmatrix}
        \boxed{1}&\boxed{1}&0&0&\boxed{1}&0\\
        \frac{b_{1,2}}{a_{1,2}}&\frac{b_{1,3}}{a_{1,3}}&\boxed{1}&\boxed{1}&\frac{b_{2,4}}{a_{2,4}}&\boxed{1}
    \end{pmatrix}\]
    and
    \[U=\begin{pmatrix}
    \boxed{a_{1,2}}&\boxed{a_{1,3}}&a_{1,4}&a_{2,3}&\boxed{a_{2,4}}&a_{3,4}\\
        0&0&\boxed{b_{1,4}-\frac{b_{2,4}}{a_{2,4}}a_{1,4}}&\boxed{b_{2,3}-\frac{b_{1,3}}{a_{1,3}}a_{2,3}}&0&\boxed{b_{3,4}-\frac{b_{1,2}}{a_{1,2}}a_{3,4}}
    \end{pmatrix},\]
    where the $S^2$-diagonal is boxed for clarity. Notice that these matrices are $S^2$-Lower Triangular and $S^2$-Upper Triangular, respectively. 
    
    These matrices are constructed from the corresponding leg matrices. For example, for $A^L_1$, which corresponds to columns $(1,3)$ and $(2,3)$, we have the matrix
    \[A^L_1=\begin{pmatrix}
        a_{1,3}&a_{2,3}\\
        b_{1,3}&b_{2,3}
    \end{pmatrix}\]
    which decomposes into matrices $U^L_1$ and $L^L_1$ by Proposition \ref{LUProp} since $A^L_1$ has nonzero leading principal minors. In particular,
    \[L^L_1=\begin{pmatrix}
        1&0\\
        \frac{b_{1,3}}{a_{1,3}}&1
    \end{pmatrix}\]
    and
    \[U^L_1=\begin{pmatrix}
        a_{1,3}&a_{2,3}\\
        0&b_{2,3}-\frac{b_{1,3}}{a_{1,3}}a_{2,3}
    \end{pmatrix}\]
    by direct computation. Using $L^L_1$ and $U^L_1$, we can construct columns $(1,3)$ and $(2,3)$ of $L$ and $U$ directly. The other columns are similar.
    
    Therefore, by using Equation \ref{S2Mult}, we have
    \begin{align*}
        LU&=\begin{pmatrix}
        1&1&0&0&1&0\\
        \frac{b_{1,2}}{a_{1,2}}&\frac{b_{1,3}}{a_{1,3}}&1&1&\frac{b_{2,4}}{a_{2,4}}&1
    \end{pmatrix}\cdot \begin{pmatrix}
    a_{1,2}&a_{1,3}&a_{1,4}&a_{2,3}&a_{2,4}&a_{3,4}\\
        0&0&b_{1,4}-\frac{b_{2,4}}{a_{2,4}}a_{1,4}&b_{2,3}-\frac{b_{1,3}}{a_{1,3}}a_{2,3}&0&b_{3,4}-\frac{b_{1,2}}{a_{1,2}}a_{3,4}
    \end{pmatrix}\\
    &\\
    &=\begin{pmatrix}
        a_{1,2}&a_{1,3}&a_{1,4}&a_{2,3}&a_{2,4}&a_{3,4}\\
        b_{1,2}&b_{1,3}&\frac{b_{2,4}}{a_{2,4}}a_{1,4}+b_{1,4}-\frac{b_{2,4}}{a_{2,4}}a_{1,4}&\frac{b_{1,3}}{a_{1,3}}a_{2,3}+b_{2,3}-\frac{b_{1,3}}{a_{1,3}}a_{2,3}&b_{2,4}&\frac{b_{1,2}}{a_{1,2}}a_{3,4}+b_{3,4}-\frac{b_{1,2}}{a_{1,2}}a_{3,4}
    \end{pmatrix}\\
    &\\
    &=A.
    \end{align*}
\end{proof}

\section*{Acknowledgment}
We thank Mihai Staic for his help related to the code for finding $det^{S^2}(A)$ for the case of $d\leq5$.

\bibliographystyle{amsalpha}

\begin{thebibliography}{A}





\bibitem
%[H]
{bour}
N. Bourbaki,    \textit{Algebra I: Chapters 1--3}, Springer Verlag, New York,  (1989).

%\bibitem{dim4} M. J. Fyfe, S. R. Lippold, M. D. Staic, and A. Stancu, \textit{Twin-star hypothesis and cycle-free $d$-partitions of $K_{2d}$}. arXiV:   (2023).

\bibitem{SDiss} S. R. Lippold, \textit{Generalizations of the Exterior Algebra}. Ph.D. Thesis, (2023).

\bibitem{dets3} S. R. Lippold and M. D. Staic, \textit{Existence of the Map $det^{S^3}$}. arXiv:2211.10375 (2022).


\bibitem 
%[S2dim3]
{edge} S. R. Lippold, M. D. Staic, and A. Stancu, \textit{Edge partitions of the complete graph and a determinant-like function}. Monatsh. Math. \textbf{198} (2022), 819--858. 

\bibitem{lu} J. Lu, \textit{Matrix Decomposition and Applications}. Eliva Press, Chisinau, Moldova, (2022).

\bibitem{dets2} M. D. Staic, \textit{Existence of the $det^{S^2}$ Map}.  Bull. of the Lon. Math. Soc., \textbf{55} (2023), no. 4, 1685--1699.

\bibitem
%[SL]
{sta2} M. D. Staic, \textit{The Exterior Graded Swiss-Cheese Operad $\Lambda^{S^2}(V)$ (with an appendix by Ana Lorena Gherman and Mihai D. Staic)}. 	 Comm. Algebra, \textbf{51} (2023), no. 7, 2705--2728.



\bibitem{sv}
M. D. Staic, and J. Van Grinsven, \textit{A geometric application for the map $det^{S^2}$}. Comm. Algebra, \textbf{50} (2022), no. 3, 1106--1117.



\end{thebibliography}

\end{document}